\documentclass[11pt]{article}

\usepackage[margin=1in]{geometry}
\usepackage{mathrsfs}
\usepackage{amsmath,amssymb,amsfonts}
\usepackage{amsthm}
\usepackage{bm}
\usepackage{xcolor}
\usepackage{graphicx}
\usepackage{float}
\usepackage{booktabs}
\usepackage{tikz}
\usepackage{pgfplots}
\pgfplotsset{compat=newest}
\usetikzlibrary{decorations.markings}
\usepackage{pgfplotstable}
\usepackage[normalem]{ulem}
\usepackage[colorlinks=true,urlcolor=blue,citecolor=red,linkcolor=blue]{hyperref}

\graphicspath{{./figures/}}

\newtheorem{theorem}{Theorem}[section]
\newtheorem{lemma}[theorem]{Lemma}

\theoremstyle{definition}
\newtheorem{definition}[theorem]{Definition}

\theoremstyle{remark}
\newtheorem{remark}[theorem]{Remark}

\numberwithin{equation}{section}
\numberwithin{figure}{section}
\def\bld{\boldsymbol}

\newcommand{\Au}{{\bm{u}}}

\allowdisplaybreaks
\renewcommand\arraystretch{1}

\title{A quasi-incompressible Cahn--Hilliard--Darcy model for two immiscible fluids in porous media}

\author{
Daozhi Han\\
\small Department of Mathematics, State University of New York at Buffalo, Buffalo, NY 14260\\
\small \texttt{daozhiha@ub.edu}
\and
Sayantan Sarkar\\
\small Department of Mathematics, State University of New York at Buffalo, Buffalo, NY 14260\\
\small \texttt{sayantan@buffalo.edu}
}

\date{}

\begin{document}

\maketitle

\begin{abstract}
We derive a quasi-incompressible Cahn-Hilliard-Darcy  phase-field model (qCHD) with the logarithmic Flory-Huggins free energy density function for two-phase flows in porous media. The model satisfies an energy-dissipation law. In the formal sharp interface limit, the qCHD model gives rise to the classical Muskat's problem.  By exploiting estimates of the pressure from Darcy's equations, we establish global existence of weak solutions in both 2D and 3D to the qCHD model.
\end{abstract}

\noindent\textbf{Keywords:} Navier-Stokes, Cahn-Hilliard, Darcy, diffuse interface model, well-posedness, superposed free flow and porous media

\medskip
\noindent\textbf{Mathematics Subject Classification (2010):} 35K61, 76T99, 76S05, 76D07.

\medskip

	\section{Introduction}
	\label{intro-sec}
The study of two-phase flows in porous media is vital for both mathematical analysis and engineering applications \cite{Yan2022} including unsaturated soil flow, petroleum reservoir displacement, geological ($\mathrm{CO}_2$) storage \cite{KOLDITZ20121919}, proton exchange membrane fuel cells \cite{WANG2008603} water management in polymer-electrolyte fuel cells \cite{WeberNewman2004}. Historically, the Muskat problem \cite{Muskat1934} has been employed as a conventional model that depicts a fluid interface as an infinitely thin discontinuity. Although this model is effective for fixed geometries, it is vulnerable to finite-time singularities: smooth initial profiles can become unbounded \cite{Siegel2004}; the Rayleigh--Taylor stability criterion may initially be satisfied but can later fail \cite{Castro2012Rayleigh}; and even analytic data in a seemingly stable regime can result in instability and collapse \cite{Castro2013BreakdownFull}. To mitigate these issues, diffuse-interface (phase-field) formulations substitute the sharp discontinuity with a narrow yet finite transition layer, offering a more robust description that naturally regularizes the dynamics. For an in-depth review, we refer to Anderson et al. \cite{anderson1998diffuse}.

Building on the diffuse interface description, Cahn and Hilliard \cite{CahnHilliard1958} introduced an evolution equation for the interface between two phases. The Cahn--Hilliard equation is derived by minimizing the free-energy potential, with the minimum values corresponding to the different phases (i.e., different fluids). The CH equation has been thoroughly investigated theoretically, both without logarithmic potential (\cite{ElliottZheng1986,BloweyElliott1991,ElliottGarcke1996,Chen1996}) and with logarithmic potential (\cite{DebusscheDettori1995,MiranvilleZelik2004,AbelsWilke2007,CherfilsMiranvilleZelik2011,LondenPetzeltova2018}). In the context of two-phase flows, the Cahn--Hilliard equation was coupled with fluid equations such as the Navier--Stokes equations. This coupling, now referred to as the Cahn--Hilliard--Navier--Stokes (CHNS) system, initially emerged in kinetic theories of critical dynamics \cite{kawasaki1970kinetic,halperin1974renormalization,siggia1976critical} and was formalized as Model H by Hohenberg and Halperin \cite{Hohenberg1977}. For the matched-density scenario ($\rho_1=\rho_2$), the CHNS model was rigorously grounded in thermodynamics using the microforce formulation by Gurtin et al. \cite{Gurtin1996}. Subsequently, Jacqmin conducted the first fully resolved simulations and formally derived the sharp interface limit as the interfacial thickness approached zero \cite{Jacqmin1999}. With a regular (polynomial) free-energy density, Boyer demonstrated global weak existence and uniqueness \cite{Boyer2001}, while Gal and Grasselli later identified exponential attractors that govern long-term dynamics \cite{GalGrasselli2010}. When employing the physically realistic Flory--Huggins logarithmic potential, Abels established the theory with a comprehensive global weak/strong solvability result \cite{Abels2009Singular}. Two variants arise for fluids with different densities ($\rho_1\neq\rho_2$). In the strictly incompressible context ($\nabla\cdot u=0$), the frame-indifferent Abels--Garcke--Gr\"un formulation provides global weak solutions for polynomial free energy \cite{AbelsGarckeGrun2012} and was extended to include degenerate mobility and non-local interactions by Abels, Depner, and Garcke \cite{AbelsDepnerGarcke2013}. Allowing for slight compressibility results in a quasi-incompressible version ($\nabla\cdot u\propto\partial_t\rho$), as derived by Lowengrub and Truskinovsky, who also established its sharp interface limit \cite{LowengrubTruskinovsky1998}. Abels later demonstrated strong well-posedness in two dimensions \cite{Abels2012Quasi}, and Abels and Marquardt justified the Stokes-type sharp interface limit for unequal densities \cite{AbelsMarquardt2021}. Recent studies incorporate logarithmic or non-local energies, notably by Frigeri, Gal, and Grasselli \cite{FrigeriGalGrasselli2016} and Knopf and Signori \cite{KnopfSignori2021}.

Phase-field models for thin-gap and porous-medium flows were developed following the derivation by Folch \emph{et al.} of a Cahn--Hilliard matched-asymptotic expansion. This expansion simplifies, after gap-averaging and applying the creeping-flow limit $Re\rightarrow 0$, to the classical Hele--Shaw free-boundary problem \cite{Folch1999}. Subsequent calibration for pinch-off and reconnection by Lee, Lowengrub, and Goodman confirmed the accuracy of this reduction \cite{LeeLowengrubGoodman2002,LeeLowengrubGoodman2002II}. More broadly, starting from the matched-density CHNS equations, the Cahn--Hilliard--Hele--Shaw/Darcy (CHHS/CHD) system can be derived either through the same slit-averaging approach or by incorporating an isotropic permeability drag into the momentum balance. This procedure is detailed in the diffuse-interface review by Anderson, McFadden, and Wheeler \cite{anderson1998diffuse} and has been rigorously formalized for variable density by Ded\`e, Garcke, and Lam, who derived CHHS from the frame-indifferent AGG CHNS model \cite{Dede2018}. For matched densities with classical polynomial free energy, global weak solutions were initially proven by Feng and Wise \cite{FengWise2012}, extended to weak--strong uniqueness in Stokes--Darcy couplings by Han, Wang, and Wu \cite{HanWangWu2014}, and further analyzed for viscosity contrast in \cite{Dede2018}. Under the physically realistic Flory--Huggins logarithmic potential, global weak solutions and two-dimensional uniqueness were established by Giorgini, Grasselli, and Wu \cite{GiorginiGrasselliWu2018}. Additionally, Jia \emph{et al.} \cite{Jia2020} and Cavaterra, Frigeri, and Grasselli \cite{CavaterraFrigeriGrasselli2021} developed decoupled energy-stable schemes, Wang and Makhlouf introduced an unconditionally stable linear time-stepping method \cite{WangMakhlouf2022}, and Giorgini provided further well-posedness results \cite{Giorgini2019WellPosedness}. For unmatched densities, the incompressible CHHS theory is grounded in the variable-density model by Ded\`e, Garcke, and Lam \cite{Dede2018}, as well as in the tumor-growth Brinkman variants by Ebenbeck, Garcke, and N\"urnberg \cite{EbenbeckGarckeNurnberg2021} and Knopf and Signori \cite{KnopfSignori2021}. Quasi-incompressible or mass-source formulations yield global strong solutions (Giorgini, Lam, Rocca, and Schimperna \cite{GiorginiLamRoccaSchimperna2020}) and existence with strongly separating logarithmic energies (Schimperna \cite{Schimperna2022}), although a complete Darcy analogue of the quasi-incompressible CHHS/CHD model remains an unresolved issue.

In this contribution, we derive a quasi-incompressible Cahn--Hilliard--Darcy model (qCHD) for two-phase incompressible fluids in porous media. The qCHD model uses the volume fraction difference as the order parameter and employs the  mass-averaged mixture velocity that is non-solenoidal. Based on Onsager's extremum principle, we systematically derive the qCHD model, which obeys local mass conservation (the continuity equation) and satisfies an energy-dissipation law. By leveraging the pressure estimate from Darcy's law, we show that the qCHD system with the Flory-Huggins logarithmic potential admits a global  weak solution in both 2D and 3D. Such results are not known to be available for quasi-incompressible Cahn-Hilliard-Navier-Stokes equations. Finally, the formal matched asymptotic expansion of the qCHD model with the physically relevant logarithmic potential gives rise to a modified Muskat problem that is close to the original one.
To the best of our knowledge, this is the first study to present a quasi-incompressible Cahn--Hilliard--Darcy system that incorporates a singular logarithmic potential and non-uniform fluid properties. Section \ref{sec(2)} details the derivation process using Onsager's variational principle, Section \ref{sec(4)} introduces weak formulations and provides a comprehensive proof of the global existence of weak solutions, and Appendix \ref{App} illustrates the alignment of the sharp interface limit with a Muskat-like problem.

	\section{The quasi-incompressible Cahn-Hilliard-Darcy system}\label{sec(2)}
	\subsection{Derivation of the quasi-incompressible Cahn-Hilliard-Darcy model}
	Let $\rho_1, \rho_2$ be the specific (constant) densities of fluid 1 and fluid 2, respectively, occupying a porous media. The order parameter is the volume fraction difference denoted by $\phi$. The density and viscosity of the mixture are then
	\begin{align*}
		\rho=\frac{1-\phi}{2}\rho_1+\frac{1+\phi}{2}\rho_2, \quad \eta=\frac{1-\phi}{2}\eta_1+\frac{1+\phi}{2}\eta_2,
	\end{align*}
	with $\eta_1$ and $\eta_2$ being the viscosities of fluid 1 and 2. From the mass conservation of each fluid one derives the continuity equation
	\begin{align}
		&\partial_t (\chi\rho)+\nabla \cdot (\rho \bld u )=0, \label{CHD_con_mass}
	\end{align}
	where $\bld u=\frac{1-\phi}{2}\rho_1 \bld u_1+\frac{1+\phi}{2}\rho_2 \bld u_2$ is the mass-averaged velocity, and $\chi$ is the porosity. The order parameter $\phi$ satisfies the conservation equation
	\begin{align}
		&\partial_t (\chi\phi) +\nabla \cdot (\phi \bld u)=-\nabla \cdot \bld J, \label{CHD_con_vol}
	\end{align}
	with $\bld J$ the chemical diffusion flux to be determined. It follows from the definition of $\rho$ and Eqs. \eqref{CHD_con_mass} and \eqref{CHD_con_vol} that
	\begin{align}
		\nabla \cdot \bld u= -\alpha  \nabla \cdot \bld J, \label{div}
	\end{align}
	where $\alpha=\frac{\rho_1-\rho_2}{\rho_1+\rho_2}$ is the Atwood number \cite{SHARP19843}.
	
	Denote the total energy by 
	\begin{align}
		E(\phi)=\int_{\Omega} \chi \big(  F(\phi)+\frac{\epsilon^\ast}{2}|\nabla \phi|^2\big)+\chi \rho g z\, dx. \label{CHD_ener}
	\end{align}

Also, for later convenience, we define
\begin{align}
    \tilde{E}(\phi) &=  F(\phi)+\frac{\epsilon^\ast}{2}|\nabla \phi|^2.\label{Alt-En}
\end{align}
$F(\phi)$ being the Flory--Huggins homogeneous free energy function,
    \begin{align*}
        F(\phi)=\frac{\theta_0}{2} \Big[(1+\phi)\ln{(1+\phi)}+(1-\phi)\ln{(1-\phi)}\Big]-\frac{\theta_c}{2}\phi^2+C,
    \end{align*}
    where $\theta_0$ is the quenching temperature, $\theta_c$ is the critical temperature, $0<\theta_0< \theta_c$, and $C$ is a constant such that $F\geq0$. We assume $\theta_0 \ll \theta_c$, which implies that the minimizers of $F$ are close to $\pm1$. The limit $\frac{\theta_0}{\theta_c}\rightarrow 0$ is known as the deep quench limit, leading to the double obstacle potential, which is the interest of this discourse. We refer to \cite{Abels2009Singular, GiorginiGrasselliWu2018, Miranville2019} for further discussions on the singular potential.\\
In the case where $\theta_0$ is close to $\theta_c$, $F$ is often approximated by a polynomial of double-well type. 
	Utilizing the conservation equations \eqref{CHD_con_mass} and \eqref{CHD_con_vol}
	one calculates
	\begin{align}
		&\frac{d}{dt}E(\phi)=\int_\Omega \mu \partial_t (\chi \phi) +\partial_t (\chi\rho) g zdx =-\int_\Omega \mu \nabla \cdot (\phi \bld u+\bld J) +\nabla \cdot (\rho \bld u) g zdx \nonumber \\
		&
		=\int_\Omega \nabla \mu \cdot (\phi \bld u+\bld J) +g \rho \bld u \cdot \bld k dx, \label{CHD_rate}
	\end{align}
	where one has assumed the following boundary conditions
	\begin{align*}
		\nabla \phi \cdot \bld n=	\bld u \cdot \bld n=\bld J \cdot \bld n=0, \text{ on } \partial \Omega,
	\end{align*}
	and $\bld k$ is the unit vector pointing upward, $\mu= f(\phi)-\epsilon^\ast\Delta \phi$ is the chemical potential with $f=F^\prime$.
	
	One then postulates the following dissipation function (motivated by Darcy's equations and the Cahn--Hilliard equation)
	\begin{align}
		D=\int_\Omega \frac{M^{-1}}{2}|\bld J|^2+\frac{1}{2}\bld u ^T\mathbb{K}^{-1}\bld u dx, \label{CHD_diss}
	\end{align}
	with $M$ the mobility function, $\mathbb{K}$ the hydraulic conductivity matrix. $\mathbb{K} = \frac{\Pi}{\eta}$ and $\Pi$ is the permeability matrix. Also $\mathbb{K}$ is symmetric \cite{Neuman1977} positive definite \cite{Allaire1991, ArbogastLehr2006}.
	
	One now introduces the Rayleighian with a Lagrange multiplier $p$ to relax the constraint \eqref{div}
	\begin{align}
		&R=\frac{d}{dt} E+ D-\int_\Omega  p (\nabla \cdot \bld u+\alpha \nabla \cdot \bld J) dx \label{CHD_Ray}\\
		&=\int_\Omega \nabla \mu \cdot (\phi \bld u+\bld J) +g \rho \bld u \cdot \bld k dx+\int_\Omega \frac{M^{-1}}{2}|\bld J|^2+\frac{1}{2}\bld u^T\mathbb{K}^{-1}\bld u dx-\int_\Omega  p (\nabla \cdot \bld u+\alpha \nabla \cdot \bld J) dx. \nonumber 
	\end{align}
	Following Onsager's extremum principle (equivalent to the principle of Rayleigh's least energy dissipation), one minimizes $R$ with respect to the rate functions $\bld u, \bld J$ and the Lagrange multiplier $p$. One obtains
	\begin{align*}
		\bld J=-M\nabla ( \mu+\alpha p),  \quad \phi \nabla \mu +g\rho \bld k+\mathbb{K}^{-1} \bld u +\nabla p=0.
	\end{align*}
	
	In summary we have derived the following quasi-incompressible Cahn--Hilliard--Darcy equations (qCHD)
	\begin{subequations}\label{CHD}
		\begin{align}
			&  \mathbb{K}^{-1} \bld u =-\nabla p-\phi \nabla \mu-\rho g \bld k, \label{CHD-1} \\
			&\nabla \cdot \bld u=\alpha \nabla \cdot \big(M\nabla \mu_p\big), \label{CHD-2}\\
			&\partial_t (\chi \phi) +\nabla \cdot (\phi \bld u)=\nabla \cdot \big(M\nabla \mu_p\big), \label{CHD-3} \\
			&\mu_p=\mu+\alpha p, \quad \mu= f(\phi)-\epsilon^\ast\Delta \phi, \label{CHD-4}\\
			&\nabla \phi \cdot \bld n=	\bld u \cdot \bld n=M \nabla \mu_p \cdot \bld n=0, \text{ on } \partial \Omega.
		\end{align}
	\end{subequations}
	A consequence of the variational procedure is that the qCHD system satisfies an energy law
	\begin{align}
		\frac{d}{dt}E(\phi)=-2D=-\int_{\Omega} M |\nabla \mu_p|^2+\bld u^T\mathbb{K}^{-1}\bld u dx. \label{CHD_ener_law}
	\end{align}

    \begin{remark}
        An equivalent formulation of the qCHD system can be derived by incorporating the gravitational effect into the chemical potential:
       \begin{subequations}\label{eq-CHD}
		\begin{align}
			&  \mathbb{K}^{-1} \bld u =-\nabla p-\phi \nabla \mu, \label{eq-CHD-1} \\
			&\nabla \cdot \bld u=\alpha \nabla \cdot \big(M\nabla \mu_p\big), \label{eq-CHD-2}\\
			&\partial_t (\chi \phi) +\nabla \cdot (\phi \bld u)=\nabla \cdot \big(M\nabla \mu_p\big), \label{eq-CHD-3} \\
			&\mu_p=\mu+\alpha p, \quad \mu= f(\phi)-\epsilon^\ast\Delta \phi+\rho^{\prime} g z, \label{eq-CHD-4}\\
			&\nabla \phi \cdot \bld n=	\bld u \cdot \bld n=M \nabla \mu_p \cdot \bld n=0, \text{ on } \partial \Omega.
		\end{align}
	\end{subequations} 
    \end{remark}

    \section{Global existence of weak solutions}\label{sec(4)}    
	\subsection{The main result}
    The dimensionless form the qCHD system is given in the Appendix, Eqs. \eqref{rescale1}--\eqref{rescale4}. Throughout, $\Omega \subset \mathbb{R}^3$ is a bounded domain. 
	We slightly modify the quasi-incompressible Cahn-Hilliard-Darcy equations (qCHD) and present them as:
	\begin{align}
		&\Pi^{-1} \eta (\phi) \bld u =-\frac{2\rho_1}{\rho_1+\rho_2}\rho \nabla p-\frac{1}{\epsilon Ca^\ast}\phi \nabla \mu_p-\frac{Bo}{Ca^\ast}\rho \bld k, \label{qCHDs1}\\
		&\partial_t (\chi \rho) +\nabla \cdot (\rho \bld u)=0, \label{qCHDs2}\\
		&\partial_t (\chi \phi) +\nabla \cdot (\phi \bld u)=\frac{1}{Pe}\Delta \mu_p, \label{qCHDs3}\\
		&\mu_p=f(\phi)-\epsilon^2\Delta \phi+\alpha \epsilon Ca^{\ast} p, \quad \rho=\frac{1-\phi}{2}+ \frac{1+\phi}{2}\frac{\rho_2}{\rho_1},\label{qCHDs4}
	\end{align}
	equipped with the boundary conditions
	\begin{align*}
		\bld u\cdot \bld n=\nabla \phi \cdot \bld n=\nabla \mu_p \cdot \bld n=0, \quad \text{ on } \Gamma:= \partial \Omega.
	\end{align*}
    Here $f(\phi)=\ln{\frac{1+\phi}{1-\phi}}-\phi$.
	Since most of the parameters do not affect the analysis in an essential way, we set them as unity, i.e. $\Pi=I, \epsilon=Ca^\ast=Bo=\chi=Pe=1$.

	Let $Q_{s, t}$ denote $\Omega \times (s, t)$ and $Q=Q_{0, T}$ for fixed $T>0$. The $L^2$ inner product and norm are denoted  by $(\cdot, \cdot)$ and $||\cdot||$, respectively.  We introduce $M:=H^1(\Omega) \cap L_0^2(\Omega)$ with $L_0^2(\Omega)$ the subspace of $L^2(\Omega)$ of mean zero; $\mathbf{X}:=[L^2(\Omega)]^3$, $Y:=H^1(\Omega)$, $Y^\prime$ the dual of $Y$, and $\langle\cdot, \cdot\rangle$ the duality between $Y^\prime$ and $Y$.  Finally the energy of the system is defined as
	\begin{align}
		E(\phi):=\int_{\Omega} F(\phi)+\frac{1}{2}|\nabla \phi|^2+\rho z \,dx. \label{Ener}
	\end{align}

	The weak formulation of the qCHD system is defined as follows.
	\begin{definition}\label{def-w}
		Assume $\phi_0 \in L^\infty(\Omega) \cap Y$ with $-1<\phi_0<1$ a.e.. Then $\left(\rho, \phi, \bld{u}, \mu_p, p\right)$ is called  a weak solution of the system \ref{qCHDs1}-\ref{qCHDs4} on time interval $[0,T]$ if
		\begin{align*}
			& \phi \in L^\infty(0,T; Y) \cap H^1(0,T; Y^\prime) \cap L^4(0,T;H^2(\Omega)), \\
			&f(\phi) \in L^2(Q), \quad \phi \in L^\infty(Q),  \quad |\phi|<1 \quad a. e., \\
			&\Au \in L^2(0,T; \mathbf{X}), \quad p \in L^2(0,T; M),\quad \mu_p \in L^2(0,T; Y), 
		\end{align*}
		and, 
		\begin{enumerate}
			\item $ \forall \varphi \in Y$, $t\in(0,T)$ a.e.
			\begin{align}
				&\langle\partial_t \phi, \varphi\rangle-(\phi \Au, \nabla \varphi)=-(\nabla \mu_p, \nabla \varphi), \label{CH1}\\
				&\langle\partial_t \rho, \varphi\rangle-(\rho \Au, \nabla \varphi)=0, \label{Con}
			\end{align}
			where  for $(x, t) \in Q$ a.e.
			\begin{align}
				&f(\phi)-\Delta \phi+\alpha p=\mu_p, \label{CH2} \\
				&\eta (\phi) \bld u =-\frac{2\rho_1}{\rho_1+\rho_2}\rho \nabla p-\phi \nabla \mu_p-\rho \bld k,  \rho=\frac{1-\phi}{2}+ \frac{1+\phi}{2}\frac{\rho_2}{\rho_1}, \label{Dar}
			\end{align}
			
			\item $\phi|_{t=0}=\phi_0$, $\nabla \phi \cdot n=0$ on $\Gamma$.
			
		\end{enumerate}
		
	\end{definition}

	The main result of this section is the following theorem.
	
	\begin{theorem}
		\label{exis-theo}
		Suppose $\phi_0 \in Y$ such that $E(\phi_0)<\infty$ and $|\phi_0|<1$ a. e. Then for any $T>0$ the qCHD system \eqref{qCHDs1}--\eqref{qCHDs4} admits a weak solution in the sense of Definition \ref{def-w}. Moreover, one has $\phi \in C(0,T; Y)$, and the following energy identity holds
		\begin{align}
			\label{En-law}
			E(\phi(t))+\int_{Q_{0,t}} \eta(\phi)|\Au|^2+|\nabla \mu_p|^2\, dxdt=E(\phi_0).
		\end{align}
		
	\end{theorem}

	In the rest of this section we establish existence of weak solutions by compactness argument. We first construct a sequence of approximate solutions from a time-marching scheme; then we establish a prior estimates of the approximate solution; finally we pass to the limit and show the limit is a weak solution.
	
	\subsection{Approximate solutions}
	To construct a sequence of approximate solutions we design a time-marching scheme and establish existence of solutions to the time-discrete system. For an integer $N>0$, denote $\delta t=\frac{T}{N}$, $t^k=k\delta t$,  and $\phi^k:= \phi(t^k, x)$.  For $k=0, 1, \ldots$, given $\phi^k \in Y\cap L^\infty(\Omega)$ with $-1<\phi^k<1$ a. e., we seek 
	\begin{align*}
		(\phi^{k+1}, \mu_p^{k+1}, p^{k+1}, \Au^{k+1}) \in Y\cap L^\infty(\Omega)\times Y \times M \times \mathbf{X} \text{ with } -1< \phi <1 \text{ a. e.}
	\end{align*}
	such that
	$ \forall \varphi \in Y$, 
	\begin{align}
		&(\phi-\phi^k, \varphi)-\delta t(\phi^k \Au, \nabla \varphi)+\delta t(\nabla \mu_p, \nabla \varphi)=0, \label{CHd1}\\
		&(\rho-\rho^k, \varphi)-\delta t(\rho^k \Au, \nabla \varphi)=0, \label{Cond}\\
		&\big(f_v(\phi)+f_n(\phi^k), \varphi\big)+(\nabla \phi, \nabla \varphi)+\alpha (p, \varphi)=(\mu_p, \varphi), \label{CHd2} \\
		&\eta (\phi^k) \bld u =-\frac{2\rho_1}{\rho_1+\rho_2}\rho^k \nabla p-\phi^k \nabla \mu_p-\rho^k \bld k,  \rho=\frac{1-\phi}{2}+ \frac{1+\phi}{2}\frac{\rho_2}{\rho_1} \quad a.e. \label{Dard}, 
	\end{align}
	where $f_v=F_v^\prime, f_n=F_n^\prime$ correspond to convex-concave splitting of $F$, i.e. $F=F_v+F_n$.

	We first establish existence of solutions to the system \eqref{CHd1}--\eqref{Dard}. We reformulate the system into the following equivalent system
	\begin{align}
		&(\phi-\phi^k, \varphi)-\delta t(\phi^k \Au, \nabla \varphi)+\delta t(\nabla \mu_p, \nabla \varphi)=0, \label{CHde1}\\
		&\big(f_v(\phi)+f_n(\phi^k), \varphi\big)+(\nabla \phi, \nabla \varphi)=(\mu, \varphi), \label{CHde2} \\
		&\eta (\phi^k) \bld u = -\nabla p-\phi^k \nabla \mu-\rho^k \bld k, \label{Darde}\\
		&(\Au, \nabla \varphi)=\alpha (\nabla \mu_p, \nabla \varphi), \label{Conde}\\
		&\mu_p=\mu+\alpha p,  \rho=\frac{1-\phi}{2}+ \frac{1+\phi}{2}\frac{\rho_2}{\rho_1} \quad a.e. \label{PW}, 
	\end{align}
	
	We follow the monotone argument from \cite{HaWa2015} to show existence of solutions to Eqs. \eqref{CHde1}--\eqref{PW}.  The following results are standard.
	\begin{lemma} \label{LeChe}
		Given $\phi^k \in Y$ and $|\phi^k| <1$ a.e., for any $\mu \in Y$, there exists a unique solution $\phi$ to Eq. \eqref{CHde2} such that
		\begin{align*}
			\int_\Omega \phi dx=\int_\Omega \phi^k dx, \quad	\phi \in H^2(\Omega), \quad |\phi| <1 \quad a. e.
		\end{align*}
		Moreover, the solution operator $\phi(\mu): Y \rightarrow H^2(\Omega)$ is bounded,  continuous in the strong topology of $Y$ and in the weak topology of $H^2(\Omega)$.
	\end{lemma}
	\begin{lemma}\label{sol2}
		Given $\phi^k \in Y$ and $|\phi^k| <1$ a.e., for any $\mu \in Y$, there exists a unique pair $(\Au, p) \in \mathbf{X} \times M$ to Eqs. \eqref{Darde} and \eqref{Conde}.  Furthermore, the linear solution operator  is bounded in the strong topology. 
	\end{lemma}
	
	The following lemma establishes existence of solutions to the time-discrete system \eqref{CHde1}--\eqref{PW}
	
	\begin{lemma}\label{ex-tds}
		Given $\phi^k \in Y$ and $|\phi^k| <1$ a.e., there existence a weak solution to the sytem \eqref{CHde1}--\eqref{PW} such that
		\begin{align*}
			\phi \in H^2(\Omega), \Au \in \mathbf{X}, p \in M, \mu \in Y, |\phi|<1 \quad a.e.
		\end{align*}
		Moreover, the solution to the time-discrete system satisfies an energy law
		\begin{align}\label{en-dis}
			&E(\phi^{n+1}) +\sum_{k=0}^{n} \frac{1}{2}||\nabla (\phi^{k+1}-\phi^k)||^2+\delta t\sum_{k=0}^{n} 
			\int_{\Omega} \eta(\phi^k)|\Au^{k+1}|^2\\
			&+\delta t\sum_{k=0}^{n} |\nabla \mu_p^{k+1}|^2 \, dx \leq E(\phi_0), \quad n=1, 2\ldots N-1. \nonumber
		\end{align}
	\end{lemma}
	\begin{proof}
		For given $\mu \in Y$, owing to Lemma \ref{LeChe} and \ref{sol2}, the left-hand side of Eq. \eqref{CHde1} defines a continous and bounded operator $T$: $Y\rightarrow Y^\prime$
		\begin{align*}
			\langle T(\mu), \nu\rangle:=(\phi-\phi^k, \nu)-\delta t(\phi^k \Au, \nabla \nu)+\delta t(\nabla \mu_p, \nabla \nu), \quad \forall \nu \in Y,
		\end{align*}
		where we recall $\mu_p=\mu+p$. 
		It follows $\forall \mu, \nu \in Y$
		\begin{align}
			&\langle T(\mu)-T(\nu), \mu-\nu\rangle= (\phi_\mu-\phi_\nu, \mu-\nu)-\delta t\big(\phi^k (\Au_\mu-\Au_\nu), \nabla (\mu-\nu)\big)\nonumber \\
			&+\delta t \big(\nabla (\mu_p-\nu_p), \nabla (\mu-\nu)\big). \label{mono}
		\end{align}
		
		We show that the right-hand side of Eq. \eqref{mono} is non-negative.
		Thanks to the convexity of $F_v$, Eq. \eqref{CHde2} implies 
		\begin{align}
			\label{mono1}
			(\phi_\mu-\phi_\nu, \mu-\nu)=||\nabla(\phi_\mu-\phi_\nu)||^2+\big(f_v(\phi_\mu)-f_v(\phi_\nu), \phi_\mu-\phi_\nu\big) \geq 0,
		\end{align}
		with equality iff $\mu=\nu$. Likewise, one derives from Eqs. \eqref{Darde} and \eqref{Conde}
		\begin{align}
			\label{mono2}
			&	-\delta t\big(\phi^k (\Au_\mu-\Au_\nu), \nabla (\mu-\nu)\big)=\delta t\int_\Omega \eta(\phi^k)|\Au_\mu-\Au_\nu|^2\, dx+\delta t\big(\nabla (p_\mu-p_\nu),(\Au_\mu-\Au_\nu)\big)  \nonumber \\
			&=\delta t\int_\Omega \eta(\phi^k)|\Au_\mu-\Au_\nu|^2\, dx+\alpha \delta t\big(\nabla(\mu_p-\nu_p), 
			\nabla (p_\mu-p_\nu)\big).
		\end{align}
		Finally, one has
		\begin{align}
			\label{mono3}
			&\delta t \big(\nabla (\mu_p-\nu_p), \nabla (\mu-\nu)\big)=\delta t || \nabla(\mu_p-\nu_p) ||^2- \delta t \alpha \big(\nabla(\mu_p-\nu_p), 
			\nabla (p_\mu-p_\nu)\big).
		\end{align}
		Collecting Eqs. \eqref{mono1}--\eqref{mono3}, one concludes
		\begin{align*}
			\langle T(\mu)-T(\nu), \mu-\nu\rangle \geq 0, \forall \mu, \nu \in Y	
		\end{align*}
		with equality iff $\mu=\nu$.  Hence the operator $T$ is strictly monotone.
		
		Next, we show that $T$ is also coercive.  One calculates
		\begin{align}
			\label{coer}
			\langle T(\mu), \mu\rangle:=(\phi-\phi^k, \mu)-\delta t(\phi^k \Au, \nabla \mu)+\delta t(\nabla \mu_p, \nabla \mu).
		\end{align}
		In a similar fashion, one deduces
		\begin{align*}
			&(\phi-\phi^k, \mu) \geq \frac{1}{2}||\nabla \phi||^2+\int_\Omega F(\phi)\, dx-\Big( \frac{1}{2}||\nabla \phi^k||^2+\int_\Omega F(\phi^k)\, dx \Big), \\
			&-\delta t(\phi^k \Au, \nabla \mu)+\delta t(\nabla \mu_p, \nabla \mu) \geq \delta t||\nabla \mu_p ||^2+\frac{\delta t}{2} \int_{\Omega} \eta(\phi^k) |\Au|^2 \, dx-C,
		\end{align*}
		where one has utilized the uniform boundedness $|\phi^k|<1$ a. e..  By using the equivalent formulation of the Darcy equations \eqref{Dard}, one estimates the pressure gradient
		\begin{align*}
			||\nabla p||\leq C|| \sqrt{\eta} \Au||+C||\nabla \mu_p||+C,
		\end{align*}
		Hence 
		\begin{align*}
			||\nabla \mu|| \leq ||\nabla \mu_p||+|\alpha| ||\nabla p|| \leq C|| \sqrt{\eta} \Au||+C||\nabla \mu_p||+C.
		\end{align*}
		It follows
		\begin{align*}
			\langle T(\mu), \mu\rangle &\geq \delta t||\nabla \mu_p ||^2+\frac{\delta t}{2} \int_{\Omega} \eta(\phi^k) |\Au|^2 \, dx+\frac{1}{2}||\nabla \phi||^2+\int_\Omega F(\phi)\, dx-C \\
			&\geq C||\nabla \mu||^2+\frac{1}{2}||\nabla \phi||^2+\int_\Omega F(\phi)\, dx-C.
		\end{align*}
		Recall from \cite[Theorem 4.3]{AbWi2007}
		\begin{align}
			\label{Sin_l2}
			||f_v(\phi)|| \leq C (||P_0 \mu||+||\phi||+1) \leq C (||\nabla \mu||+||\nabla \phi||+1),
		\end{align}
		where $P_0$ is the $L^2$ projection onto $L^2_0(\Omega)$, and Poincare inequality is utilized in the last inequality.  Eq. \eqref{CHde2} then implies 
		\begin{align*}
			|\int_\Omega \mu \, dx| \leq C (||\nabla \mu||+||\nabla \phi||+1), 
		\end{align*}
		hence 
		\begin{align}
			\label{muH1}
			||\mu||_Y \leq C (||\nabla \mu||+||\nabla \phi||+1).
		\end{align}
		Therefore
		\begin{align*}
			\langle T(\mu), \mu\rangle &\geq C ||\mu||_Y^2-C,
		\end{align*}
		whence
		\begin{align*}
			\frac{	\langle T(\mu), \mu\rangle }{||\mu||_Y} \rightarrow \infty, \text{ as }  ||\mu||_Y \rightarrow \infty.
		\end{align*}
		This establishes the coercivity of $T$. 
		
		It follows from Browder-Minty lemma \cite{Zeidler1986} that there exists a solution $\mu \in Y$ such that 
		\begin{align*}
			0=	\langle T(\mu), \varphi \rangle=(\phi-\phi^k, \varphi)-\delta t(\phi^k \Au, \nabla \varphi)+\delta t(\nabla \mu_p, \nabla \varphi), \quad \forall \varphi \in Y.
		\end{align*}
		
		To derive the energy inequality \eqref{en-dis}, one recalls the simple identity
		\begin{align*}
			F(\phi)-F(\phi^k) \leq (f_v(\phi)-f_n(\phi^k))(\phi-\phi^k).
		\end{align*}
		One then takes $\varphi=\mu_p$ in Eq. \eqref{CHd1}, $\varphi=\phi-\phi^k$ in Eq. \eqref{CHd2}, tests Eqs. \eqref{Dard} by $\delta t \Au$, takes $\varphi=p$ and $\varphi=z$ respectively in Eq. \eqref{Cond}. This completes the proof.
	\end{proof}
	
	Now we are in a position to introduce the approximate solutions.  For $t \in [t^k, t^{k+1})$
	\begin{align*}
		\begin{aligned}
			&\phi^{\delta}:=\frac{t_{k+1}-t}{\delta t}\phi^{k}+\frac{t-t_{k}}{\delta t}\phi^{k+1},\quad  \hat{\phi}^{\delta}:=\phi^{k+1}, \quad 	\tilde{\phi}^{\delta}:=\phi^{k} \\
			&\Au^{\delta}= \Au^{k+1}, \quad  p^{\delta}:=p^{k+1}, \quad   \mu_p^{\delta}:=\mu_p^{k+1}, \\
			&\rho(\phi)=\frac{1-\phi}{2}+ \frac{1+\phi}{2}\frac{\rho_2}{\rho_1},  \quad \rho^\delta=\rho(\phi^{\delta}).
		\end{aligned}
	\end{align*}
	Then the approximate solutions satisfy the following equations $\forall \varphi \in Y$
	\begin{align}
		&(\partial_t \phi^\delta, \varphi)-(\tilde{\phi}^{\delta} \Au^{\delta}, \nabla \varphi)+(\nabla \mu_p^{\delta}, \nabla \varphi)=0, \label{CHc1}\\
		&(\partial_t \rho^\delta, \varphi)-\big(\rho(\tilde{\phi}^{\delta})  \Au^\delta, \nabla \varphi\big)=0, \label{Conc}\\
		&f_v(\hat{\phi}^\delta)+f_n(\tilde{\phi}^{\delta})-\Delta \hat{\phi}^\delta+\alpha p^\delta =\mu_p^\delta, \quad a.e., \label{CHc2} \\
		&\eta (\tilde{\phi}^{\delta}) \bld u^\delta =-\frac{2\rho_1}{\rho_1+\rho_2}\rho(\tilde{\phi}^{\delta}) \nabla p^\delta-\tilde{\phi}^{\delta} \nabla \mu_p^\delta-\rho(\tilde{\phi}^{\delta}) \bld k,  \quad a.e. \label{Darc}, \\
		&\nabla \hat{\phi}^\delta \cdot n=0, \text{ on } \Gamma, \label{CHc2-bc}
	\end{align}
	where Eqs. \eqref{CHc2} and \eqref{CHc2-bc} are due to the $H^2$ regularity of $\phi^\delta$ from Lemma \ref{ex-tds}.
	
	\subsection{Passage to the limit}

	We summarize the apriori estimates of the approximate solution in the following lemma.
	\begin{lemma}\label{apri-es}
		Assume $\phi_0 \in Y$, $E(\phi_0)<\infty$ and $|\phi_0|<1$ a.e.. Then the following estimates hold
		\begin{align}
			&|\hat{\phi}^\delta|<1, \quad a. e., \label{unif-es}\\
			& ||\hat{\phi}^\delta||_{L^\infty(0,T; Y)}+|| \Au^\delta||_{L^2(0T; \mathbf{X})}+||\mu_p^\delta ||_{L^2(0,T; Y)} \leq C_T, \label{ener-es} \\
			&||p^\delta||_{L^2(0,T; M)} \leq C_T,  \quad ||\partial_t \phi^\delta ||_{L^2(0,T;Y^\prime)} \leq C_T, \label{pre-es} \\
			&||\hat{\phi}^\delta||_{L^4(0,T; H^2(\Omega))} \leq C_T, \quad ||f_v(\hat{\phi}^\delta)||_{L^2(Q)} \leq C_T, \label{H2-es}
		\end{align}
		where the constants $C_T$ are independent of $\delta t$.
	\end{lemma}
	We note similar estimates hold for $\tilde{\phi}^\delta$ as well.
	\begin{proof}
		The proof of these estimates are standard. The uniform estimate \eqref{unif-es} is from Lemma \ref{ex-tds}. The energy law \eqref{en-dis} implies
		\begin{align*}
			||\hat{\phi}^\delta||_{L^\infty(0,T; Y)}+|| \Au^\delta||_{L^2(0T; \mathbf{X})}+||\nabla \mu_p^\delta ||_{L^2(Q)} \leq C_T.
		\end{align*}
		In light of the estimate \eqref{muH1} one then derives $|| \mu_p^\delta ||_{L^2(0,T; Y)} \leq C_T$. This establishes \eqref{ener-es}.  By using Eqs. \eqref{Darc} and \eqref{CHc1} respectively
		\begin{align*}
			&||\nabla p^\delta|| \leq C(||\Au^\delta||+||\nabla \mu_p^\delta||+1), \\
			&|(\partial_t \phi^\delta, \varphi)| \leq (||\Au^\delta||+||\nabla \mu_p^\delta||) ||\nabla \varphi||, \forall \varphi \in Y,
		\end{align*}
		one readily derives \eqref{pre-es}. Finally by testing Eq. \eqref{CHc2} with $\Delta \hat{\phi}^\delta$ and perform integration by parts one has
		\begin{align*}
			||\Delta \hat{\phi}^\delta||^2+\int_{\Omega} f_v^\prime (\hat{\phi}^\delta) |\nabla \hat{\phi}^\delta|^2\, dx \leq C( ||\nabla p^\delta|| +||\nabla \mu_p^\delta||)||\nabla \hat{\phi}^\delta||.
		\end{align*}
		The estimates \eqref{H2-es} then follow from elliptic regularity and the inequality \eqref{Sin_l2}. This completes the proof.
	\end{proof}

	We are now ready to pass to the limit and prove the main Theorem \ref{exis-theo}.
	
	\begin{proof}
		Since
		\begin{align*}
			\hat{\phi}^\delta-\phi^\delta=\frac{t^{k+1}-t}{\delta t} (\phi^{k+1}-\phi^k), \quad t \in [t^k, t^{k+1}),
		\end{align*}
		it follows $\phi^\delta$ satisfies the same estimates as $\hat{\phi}^\delta$ in Lemma \ref{apri-es}.  Moreover in view of \eqref{en-dis}
		\begin{align*}
			\|\nabla(\hat{\phi}^{\delta}-\phi^{\delta})\|_{L^2(Q)}^{2}=\frac{\delta t}{3}\sum_{k=0}^{N-1}\|\nabla(\phi^{k+1}-\phi^{k})\|_{L^{2}(Q)}^2\leq C\delta t\rightarrow 0~ as~ \delta\rightarrow 0.
		\end{align*}
		Therefore the sequences $\{\tilde{\phi}^{\delta}\}$, $\{\phi^{\delta}\}$ and $\hat{\phi}^{\delta}$, if convergent, converge to the same limit.
		
		The estimates in Lemma \ref{apri-es} implies the following weak convergence results
		\begin{align}
			&\hat{\phi}^{\delta}\longrightarrow \phi~~~~~weakly~ *~in~ L^{\infty}(0,T;Y),\label{limit_7}\\
			&~~~~~~~~~~~~~~~~weakly~in~ L^4(0,T;H^2(\Omega)),\label{limit_8}\\	
			&\partial_t \phi^\delta \longrightarrow \partial_t \phi ~~~~~weakly~in~ L^{2}(0,T;Y^\prime),\label{limit_3} \\
			& {\Au}^{\delta}\longrightarrow \Au~~~~weakly~ in~ L^2(0,T;\bm{X}),\label{limit_2}\\
			&p^{\delta}\longrightarrow p~~~~~weakly~ in ~L^{2}(0,T;M),\label{limit_5}\\
			&\mu_p^{\delta}\longrightarrow \mu_p~~~weakly~in~L^2(0,T;Y).\label{limit_11}
		\end{align}
		Furthermore, the Aubin-Lions-Simon lemma (cf. \cite[Corollary 4]{Simon1987}) yields
		\begin{align}\label{stong_phi}
			& \phi^{\delta} \longrightarrow \phi  ~~~~ strongly~in ~L^4\big(0,T;H^1(\Omega)\big) \cap C\big(0,T;L^2(\Omega)\big),
		\end{align}
		and in particular 
		\begin{align}\label{point-es}
			\phi^\delta \rightarrow \phi, \quad a.e. \text{ in } Q.
		\end{align}
		
		To pass to the limit in the nonlinear potential $f_v(\hat{\phi}^\delta)$, one introduces the set (cf. \cite{Miranville2019})
		\begin{align*}
			S_M:=\{(x, t) \in Q:  |\hat{\phi}^\delta| >1-\frac{1}{M}\}, \quad M>0 \text{ large}.
		\end{align*}
		Thanks to \eqref{H2-es} and the monotonicity of $f_v$
		\begin{align*}
			C_T\geq \int_{Q} |f_v(\hat{\phi}^\delta)| \, dxdt&\geq \int_{S_M} |f_v(\hat{\phi}^\delta)| \, dxdt \geq
			|S_M| f_v(1-\frac{1}{M}),
		\end{align*}
		hence
		\begin{align*}
			|S_M| \leq \frac{C_T}{f_v(1-\frac{1}{M})}.
		\end{align*}
		Let $\delta t \rightarrow 0$ and $M\rightarrow \infty$ respectively, one concludes that
		\begin{align*}
			|\{(x, t) \in Q:  |\phi| \geq 1\}|=0,
		\end{align*}
		that is $|\phi|<1$ a.e. in $Q$. Therefore 
		\begin{align*}
			f_v(\hat{\phi}^\delta) \rightarrow f_v(\phi) \quad a.e. \text{ in } Q.
		\end{align*}
		In light of the estimate \eqref{H2-es} again, one obtains
		\begin{align}\label{pot-conv}
			f_v(\hat{\phi}^\delta) \rightarrow  f_v(\phi)  \text{ weakly in } L^2(Q).
		\end{align}
		
		These weak and strong convergence results allow one to pass to the limit in the approximate system \eqref{CHc1}--\eqref{CHc2}, and establish the existence of weak solution in the sense of Definition \ref{def-w}.
		
		To establish the energy law, we take $\varphi=\mu_p$ in Eq. \eqref{CH1}, test Eqs. \eqref{Dar} with $\Au$, take $\varphi=\frac{2\rho_1}{\rho_1+\rho_2}p$ and $\varphi=z$ in Eq. \eqref{Con} respectively, and take summation of the results
		\begin{align}\label{En-law1}
			\langle \partial_t \phi, \mu_p \rangle +\frac{2\rho_1}{\rho_1+\rho_2} \langle \partial_t \rho, p \rangle +\langle \partial_t \rho, z\rangle=-\int_{\Omega} \eta(\phi)|\Au|^2+|\nabla \mu_p|^2\, dx.
		\end{align}
		Since 
		\begin{align*}
			\partial_t \rho=\frac{1}{2}\frac{\rho_2-\rho_1}{\rho_1} \partial_t \phi,
		\end{align*}
		and in light of $p\in L^2(0,T; M)$, one expresses Eq.\eqref{En-law1} as
		\begin{align}
			\label{En-law2}
			\langle \partial_t \phi, \mu \rangle +\langle \partial_t \rho, z\rangle=-\int_{\Omega} \eta(\phi)|\Au|^2+|\nabla \mu_p|^2\, dx,
		\end{align}
		where $\mu=f(\phi)-\Delta \phi $ is in  $L^2(0,T; Y)$.  Recall the convex-concave  splitting $ F(\phi)=F_v(\phi)+F_n(\phi)$. Following the theory of monotone operator \cite[Proof of Theorem 6]{Abels2009}, one has
		\begin{align*}
			\frac{d}{dt} \int_{\Omega}F(\phi)+\frac{1}{2}|\nabla \phi|^2\, dx=\langle \partial_t \phi, \mu \rangle. 
		\end{align*}
		Hence
		\begin{align*}
			\frac{d}{dt} E=-\int_{\Omega} \eta(\phi)|\Au|^2+|\nabla \mu_p|^2\, dx.
		\end{align*}
		Integration on $(0,t)$ gives the energy identity
		\begin{align*}
			E(\phi(t))+\int_{Q_{0,t}} \eta(\phi)|\Au|^2+|\nabla \mu_p|^2\, dxdt=E(\phi_0).
		\end{align*}
		
		Note that the energy law \eqref{En-law} implies $E(\phi(t))$ is continuous on $[0, T]$. Since $F_v$ is convex and $\phi \in C(0, T; L^2(\Omega))$, one has
		\begin{align*}
			\limsup_{s\rightarrow t} \frac{1}{2}||\nabla \phi (s)||^2 &\leq E(\phi(t))-\int_\Omega F(\phi)\, dx-\int_\Omega \rho z\, dx=\frac{1}{2}||\nabla \phi (t)||^2\\
			&\leq \liminf_{s\rightarrow t} \frac{1}{2}||\nabla \phi (s)||^2.
		\end{align*}
		Hence $\phi \in C(0,T; Y)$. This completes the proof of Theorem \ref{exis-theo}.
	\end{proof}

\section{Appendix} \label{App}
Here, we  obtain the formal sharp-interface limit of the qCHD system as the interface thickness tends to zero.
	\subsection{Non-dimensionalization}
	Following the non-dimensionalization in \cite{KAD2007, HaWa2016}, one chooses the characteristic length scale $L^\ast$, velocity scale $U^\ast$, viscosity scale $\nu^\ast$, the chemical potential scale $\theta_c$, mobility scale $m^\ast$, the permeability scale $\Pi^\ast$,  pressure scale $p^\ast=\frac{L^\ast \nu^\ast U^\ast}{\Pi^\ast}$. Define the interface thickness $\xi=\sqrt{\frac{\epsilon^{\ast}}{\theta_c}}$, and surface tension $\gamma=\frac{2\sqrt{2}\epsilon^{\ast}}{3\xi}$. 

 The non-dimensional form of the qCHD equations is as follows:
    \begin{subequations}
		\begin{align}
			\mathbb{K}^{-1} \boldsymbol{u} &= -\nabla p - \frac{1}{\epsilon Ca^\ast}\phi \nabla \mu - \frac{Bo}{Ca^\ast}\rho \boldsymbol{k},\label{rescale1} \\
			\nabla \cdot \boldsymbol{u} &= \frac{\alpha}{Pe} \nabla \cdot \left(m(\phi)\nabla \mu_p\right),\label{rescale2} \\
			\partial_t (\chi \phi) + \nabla \cdot (\phi \boldsymbol{u}) &= \frac{1}{Pe}\nabla \cdot \left(m(\phi)\nabla \mu_p\right),\label{rescale3} \\
			\mu_p &= \mu + \alpha \epsilon Ca^{\ast} p, \quad \mu = \frac{\beta}{2}\ln{\left(\frac{1+\phi}{1-\phi}\right)} -\phi-\epsilon^2\Delta \phi,\label{rescale4}\\
			\nabla \phi \cdot \boldsymbol{n} &= \boldsymbol{u} \cdot \boldsymbol{n} = m(\phi) \nabla \mu_p \cdot \boldsymbol{n} = 0, \text{ on } \partial \Omega,
		\end{align}
	\end{subequations}
	where $Da=\frac{\Pi^{\ast}}{(L^\ast)^2}$ is the Darcy number, $Pe=\frac{L^\ast U^\ast}{\theta_c m^\ast}$ is the Peclet number, $Ca=\frac{2\sqrt2\nu^{\ast}U^{\ast}}{3\gamma}$ is the capillary number, $Ca^\ast=\frac{Ca}{Da}$ is the ratio of the capillary number to the Darcy number, $\epsilon=\frac{\xi} {L^\ast}$ is the Cahn number,  $\alpha=\frac{\rho_1-\rho_2}{\rho_1+\rho_2}$ is the Atwood number, $Bo=\frac{\rho_1 (L^\ast)^2 g}{\gamma}$ is the Bond number, $\rho(\phi) = \frac{1-\phi}{2} + \frac{1+\phi}{2}\frac{\rho_2}{\rho_1}$ is the density,  $\eta(\phi) = \frac{1-\phi}{2} + \frac{1+\phi}{2}\frac{\eta_2}{\eta_1}$ is the viscosity, $\mathbb{K}=\frac{\Pi}{\eta(\phi)}$ is the dimensionless hydraulic conductivity matrix, and $\beta = \frac{\theta_0}{\theta_c}$.

    To obtain a non-trivial order-one dynamical equation for small $\epsilon$, we take 	$Pe=O(\frac{1}{\epsilon})$ and re-scale $\frac{1}{\epsilon}\mu\rightarrow \mu$, $\frac{1}{\epsilon}\mu_p\rightarrow \mu_p$.
	We obtain from \eqref{rescale1}--\eqref{rescale4}
	\begin{subequations}
	\begin{align}
			&\mathbb{K}^{-1} \bld u =-\nabla p-\frac{1}{Ca^\ast}\phi \nabla \mu-\frac{Bo}{Ca^\ast}\rho \bld k,  \label{rescale2.1} \\
			&\nabla \cdot \bld u=\alpha \epsilon^2 \nabla \cdot \big(m(\phi)\nabla \mu_p\big)\label{rescale2.2},  \\
			&\partial_t (\chi \phi) +\nabla \cdot (\phi \bld u)=\epsilon^2\nabla \cdot \big(m(\phi)\nabla \mu_p\big),\label{rescale2.3}  \\
			&\mu_p=\mu+\alpha Ca^{\ast}  p, \quad \mu= \frac{1}{\epsilon}\left[\frac{\beta}{2}\ln{\left(\frac{1+\phi}{1-\phi}\right)}-\phi\right]-\epsilon\Delta \phi \label{rescale2.4}.
		\end{align}
	\end{subequations}	
	Note that
		\begin{equation}\mu \nabla \phi = \nabla \tilde{E}-\epsilon \nabla \cdot(\nabla \phi \otimes \nabla \phi), \end{equation}
        where from \eqref{Alt-En}, $\tilde{E}=\frac{1}{\epsilon} F(\phi)+\frac{\epsilon}{2}|\nabla \phi|^2$. One can re-write Eq. \eqref{rescale2.1} as
        \begin{equation}\label{mod-darcy}
\mathbb{K}^{-1}\bld{u} = -\nabla \tilde{p} - \frac{\epsilon}{Ca^*}\nabla \cdot(\nabla \phi \otimes \nabla \phi)-\frac{Bo}{Ca^\ast}\rho \bld k,
\end{equation}
where $\tilde{p} = p +\frac{1}{Ca^*}(\phi \mu - \tilde{E})$ is the modified pressure.

\subsection{Outer expansion}	
Let the velocity $\boldsymbol{u}$, the order parameter $\phi$, and the chemical potential $\mu$ have regular expansions in the outer region:
\begin{equation}
	f = f_0 + \epsilon f_1 + \cdots, \quad \text{for } f \in \{\boldsymbol{u}, \phi, \mu\},
\end{equation}
while the modified pressure $\tilde{p}$ assumes a singular expansion:
\begin{equation}
	\tilde{p} = \frac{1}{\epsilon}\tilde{p}_{-1} +\tilde{p}_0 + \cdots.
\end{equation}
For the logic behind this choice, the reader is referred to \cite{LeeLowengrubGoodman2002,Dede2018,Fei2017,GarckeStinner2006}, where it is established that the different pressure scaling arises because the capillary stress, while regular in the outer region, becomes singular in the inner layer. This necessitates an $\mathcal{O}(\epsilon^{-1})$ pressure contribution to ensure its inner gradient balances the leading-order Korteweg stress, a standard scaling in sharp-interface asymptotics for Hele--Shaw--Cahn--Hilliard type systems. It follows from \eqref{mod-darcy} that $\tilde{p}_{-1} \equiv Const$.

From the leading-order equation $\eqref{rescale2.4}_{\text{outer}}^{(1)}$:
\begin{equation}\label{outer_1}
    \frac{\beta}{2}\ln{\left(\frac{1+\phi_0}{1-\phi_0}\right)} - \phi_0 = 0.
\end{equation}
The left-hand side is an odd function. We reject the trivial solution $\phi_0=0$ since we anticipate a two-phase sharp interface model in the limit. The other two roots $\phi_0^{\pm}$ represent each bulk phase. Here we consider the ``deep quench" regime $\beta<<1$. In this regime, the two roots can be written as $\phi_0^{\pm}=\pm(1-\delta)$ with $\delta<<1$ satisfying
\begin{equation}\label{outer_1_exp}
    \delta = 2e^{-2/\beta} \exp\left[ \delta\left(\frac{2}{\beta} - \frac{1}{2}\right) \right] \approx 2e^{-2/\beta}.
\end{equation}

In each bulk phase, one recovers the Darcy's equations:
\begin{align}
    \mathbb{K}^{-1}\boldsymbol{u_0} &= -\nabla \tilde{p}_0 -\frac{Bo}{Ca^\ast}\rho_0 \boldsymbol{k} \quad \text{in } \, \Omega^+\cup \Omega^-,\label{outer_2}\\
    \nabla \cdot \boldsymbol{u_0} &= 0 \quad \text{in} \, \Omega^+\cup \Omega^-\label{outer_3}
\end{align}
with $\rho_0 = \frac{1-\phi_0}{2} + \frac{1+\phi_0}{2}\frac{\rho_2}{\rho_1}.$

\subsection{Inner Expansion}

To analyze the transition layer, or inner region, near the evolving interface
$\Sigma(t)$, we introduce a stretched coordinate system. Let the interface
$\Sigma(t)$ be parameterized locally by $\boldsymbol{\alpha}(t,s)$, where
$s \in V \subset \mathbb{R}^{d-1}$ denotes the parameter domain of this
coordinate chart on $\Sigma(t)$. Any point $\boldsymbol{x}$ in a tubular
neighborhood of the interface can be expressed in terms of its signed distance
$\varepsilon z$ along the unit normal vector $\boldsymbol{\nu}(t,s)$, which
points into region $\Omega^-$:
\begin{equation}
  \boldsymbol{x}
  =
  \boldsymbol{\alpha}(t, s) + \varepsilon z\,\boldsymbol{\nu}(t, s),
  \qquad
  s \in V,\quad z \in (-\delta/\varepsilon,\;\delta/\varepsilon).
\end{equation}
In particular, $z = d(t,\boldsymbol{x})/\varepsilon$ is the stretched normal coordinate,
where $d(t,\boldsymbol{x})$ denotes the signed distance of $\boldsymbol{x}$ from $\Sigma(t)$. Any scalar
function $f(t,\boldsymbol{x})$ defined near the interface can then be rewritten in these
coordinates as
\[
  F(t,s,z)
  = f\bigl(t,\;\boldsymbol{\alpha}(t,s)
             + \varepsilon z\,\boldsymbol{\nu}(t,s)\bigr).
\]
This change of variables induces the corresponding transformations of the
standard differential operators, summarized in
Table~\ref{table-diff-ops}. For the derivation, the reader is referred to Appendix B-D of \cite{GarckeStinner2006}. 

\begin{table}[ht]
  \centering
  \renewcommand{\arraystretch}{1.5}
  \begin{tabular}{c | l | c | l}
    \hline
    \textbf{Operation} & \textbf{Expression} & \textbf{Operation} & \textbf{Expression} \\
\hline
    
    $\dfrac{\partial f}{\partial x}$ &
      $=-\dfrac{1}{\varepsilon} \dfrac{\partial \alpha_{2}}{\partial s}
        \dfrac{\partial F}{\partial z}+\cdots$ &
    $\dfrac{\partial f}{\partial y}$ &
      $=\dfrac{1}{\varepsilon} \dfrac{\partial \alpha_{1}}{\partial s}
        \dfrac{\partial F}{\partial z}+\cdots$ \\
    $\dfrac{\partial f}{\partial t}$ &
      $=-\dfrac{1}{\varepsilon} \boldsymbol{V^\nu} \dfrac{\partial F}{\partial z}+\cdots$ &
    $\nabla_{x} f$ &
      $=\dfrac{1}{\varepsilon} \boldsymbol{\nu} \dfrac{\partial F}{\partial z}
        +\nabla_{\Sigma} F+\cdots$ \\
    $\nabla_{x} \cdot \boldsymbol{u}$ &
      $=\dfrac{1}{\varepsilon} \dfrac{\partial \boldsymbol{V}}{\partial z}
        \cdot \boldsymbol{\nu}+\nabla_{\Sigma} \cdot \boldsymbol{V}+\cdots$ &
    $\nabla^{2} f$ &
      $=\dfrac{1}{\varepsilon^{2}} \dfrac{\partial^{2} F}{\partial z^{2}}
        -\dfrac{1}{\varepsilon} \kappa \dfrac{\partial F}{\partial z}+\cdots$ \\  
\hline
  \end{tabular}
  \caption{Differential operations and their expressions in inner coordinates.}
  \label{table-diff-ops}
\end{table}
Here, $\boldsymbol{V^\nu}$, $\nabla_{\Sigma} F$, and $\kappa$ denote the normal velocity of the interface ($\Sigma$), the surface gradient of $F$, and the mean curvature (defined as $=-\nabla_{\Sigma}\cdot \boldsymbol{\nu}$) of the interface, respectively.

The variables $(\boldsymbol{u}, \tilde{p}, \phi, \mu, \mu_p)_\epsilon$ are denoted by
$(\boldsymbol{V}, P, \Phi, M, M_p)_\epsilon$ in the inner region, and the inner
expansions are
  \[
    F_{\varepsilon}(t, s, z)
    = F_0(t, s, z) + \varepsilon F_1(t, s, z) + \cdots
    \quad\text{for } F_{\varepsilon} \in
    \{\boldsymbol{V}_{\varepsilon}, \Phi_{\varepsilon}, M_{\varepsilon}\},
  \]
  \[
    P_{\varepsilon}(t, s, z)
    = \frac{1}{\varepsilon} P_{-1}(t, s, z) + P_0(t, s, z) + \cdots  \quad\text{also for } M_p.
  \]

The inner and outer expansions are connected via matching conditions,
requiring the inner solution in the far field ($z \to \pm\infty$) to match the
outer solution at the interface ($\boldsymbol{x}\in\Sigma$):
\begin{subequations}
  \begin{align}
    \lim_{z \to \pm \infty} F_0(t, s, z)
    &= f_{0}^{\pm }(t, x),\label{matching1}\\
    \lim_{z \to \pm \infty} \frac{\partial}{\partial z}F_0(t, s, z)
    &= 0,\label{matching2}\\
    \lim_{z \to \pm \infty} \frac{\partial}{\partial z}F_1(t, s, z)
    &= \frac{\partial}{\partial \nu}f_{0}^{\pm}(t, x),\label{matching3}\\
    \lim_{z \to \pm \infty} P_{-1}(t, s, z)
    &= 0,\label{matching4}\\
    \lim_{z \to \pm \infty} P_0(t, s, z)
    &= \tilde{p}_{0}^{\pm}(t, x),\label{matching5}
  \end{align}
\end{subequations}
where
\[
  f_0^{\pm}=
  \lim_{\delta \to 0}f_0(t,x\pm\delta \boldsymbol{\nu})
  \quad \text{for } x \in \Sigma,\]

and we have assumed the constant outer pressure $\tilde{p}_{-1}$ to be zero in the matching condition \eqref{matching4}, without loss of generality.

From $\eqref{rescale2.2}_{\text{inner}}^{-1}$ one obtains
\begin{align*}
  \frac{1}{\epsilon}\frac{\partial \bld{V_0}}{\partial z}\cdot \boldsymbol{\nu} &= \mathcal{O}(1),\\
  \intertext{then to the leading-order approximation and integrating in $z$ gives}
  \bld{V_0}\cdot \boldsymbol{\nu} &= h_1(s,t).
\end{align*}
Using the matching condition \eqref{matching1} for the velocity then yields
\begin{equation}
  \left[\boldsymbol{u}_0\right]_1^2\cdot\boldsymbol{\nu}=0
  \quad \text{on } \Sigma.
\end{equation}
Here, $$\left[\boldsymbol{u}_0\right]_1^2 := \lim_{\delta \to 0}\boldsymbol{u}_0(t,x+\delta \boldsymbol{\nu}) - \lim_{\delta \to 0}\boldsymbol{u}_0(t,x-\delta \boldsymbol{\nu})$$ is the jump of $\boldsymbol{u}_0$ across the interface.

From \eqref{rescale2.2} and \eqref{rescale2.3} we obtain
\begin{equation}\label{comb}
  \chi \frac{\partial \phi}{\partial t}
  +\nabla \phi \cdot \bld{u}
  =\epsilon^2(1-\alpha\phi)\nabla\cdot\left[m(\phi)\nabla\mu_p\right].
\end{equation}
Taking the leading-order terms, $\eqref{comb}_{\text{inner}}^{-1}$ gives
\begin{equation}
  \boldsymbol{V^\nu}_0
  = \frac{1}{\chi}\bld{V_0}\cdot \boldsymbol{\nu}
  \quad \text{on } \Sigma,
\end{equation}
so that $\boldsymbol{V^\nu}_0$ is the normal interface velocity at leading order.

From $\eqref{rescale2.4}_{\text{inner}}^{-1}$ we obtain
\begin{align*}
  F'(\Phi_{0})
  &= \frac{\partial^2 \Phi_0}{\partial z^2},\\
  \intertext{where $F'(\phi)
   = \frac{\beta}{2}\ln{\frac{1+\phi}{1-\phi}} - \phi$.
   Multiplying by $\partial_z \Phi_0$ and integrating yields}
  2F(\Phi_0)
  &= \left(\frac{\partial \Phi_0}{\partial z}\right)^2 + h_2(s,t).
\end{align*}
Taking $z\rightarrow \pm \infty$ and applying the matching conditions shows
that $h_2(s,t) \equiv Const$, which implies that the leading-order profile $\Phi_{0}$
is a function of $z$ only. With
\[
F(\Phi_0)
=
\frac{\beta}{2}\left[(1+\Phi_0)\ln(1+\Phi_0)
                   +(1-\Phi_0)\ln(1-\Phi_0)\right]-\frac{\Phi_0^2}{2} +C,
\]
we find the implicit expression
\begin{equation}
  z = \pm\int_{\Phi^{\ast}}^{\Phi_{0}}\frac{d\Phi_0}{\sqrt{2F(\Phi_0)}} = \pm Z(\Phi_0) - C_1,
\end{equation}
which determines a diffuse interface profile $\Phi_0(z)$ by the implicit function theorem.

Next we analyze the modified Darcy equation \eqref{mod-darcy}. The term
$\epsilon\nabla\cdot(\nabla \phi \otimes \nabla\phi)$ transforms in inner
coordinates as (P. 542, \cite{Dede2018})
\[\left[\frac{1}{\varepsilon^{2}} \frac{\partial}{\partial z}
   \left(\frac{\partial \Phi }{\partial z}\right)^{2} \boldsymbol{\nu}
  +\frac{1}{\varepsilon} \frac{\partial}{\partial z}
   \left(\frac{\partial \Phi}{\partial z} \nabla_{\Sigma} \Phi\right)
  +\frac{1}{\varepsilon} \nabla_{\Sigma} \cdot
   \left(\left(\frac{\partial \Phi}{\partial z}\right)^{2}
          \boldsymbol{\nu} \otimes \boldsymbol{\nu}\right)\right] + \cdots.\]
Then $\eqref{mod-darcy}_{\text{inner}}^{-2}$ yields
\begin{align*}
  \frac{\partial P_{-1}}{\partial z}
  = -\frac{1}{Ca^{*}}\frac{\partial }{\partial z}
    \left(\frac{\partial \Phi_0}{\partial z}\right)^2.
\end{align*}
Integrating and using the matching conditions
$\lim_{z \to \pm \infty}P_{-1}=0$ and
$\lim_{z \to \pm \infty} \frac{\partial \Phi_0}{\partial z} = 0$, we obtain
\begin{equation}
  P_{-1} = -\frac{1}{Ca^{*}}\left(\frac{\partial\Phi_0}{\partial z}\right)^2.
\end{equation}

At the next order, $\eqref{mod-darcy}_{\text{inner}}^{-1}$ gives
\begin{align*}
  \frac{\partial P_0}{\partial z}\boldsymbol{\nu}
  + \nabla_{\Sigma}P_{-1}
  +\frac{1}{Ca^{\ast}}\frac{\partial}{\partial z}
    \left(\frac{\partial \Phi_0}{\partial z}\nabla_{\Sigma}\Phi_0\right)
  +\frac{1}{Ca^{\ast}}\nabla_{\Sigma} \cdot
    \left[\left(\frac{\partial \Phi_0}{\partial z}\right)^2
           \boldsymbol{\nu}\otimes\boldsymbol{\nu}\right]=0.
\end{align*}
Since $\Phi_0$ depends only on $z$, we have $\nabla_{\Sigma}\Phi_0=0$, and
taking the dot product with $\boldsymbol{\nu}$ and using
$\nabla_\Sigma \cdot \boldsymbol{\nu} = -\kappa$ (the mean curvature), this reduces
to
\begin{align*}
  \frac{\partial P_0}{\partial z}
  = \frac{\kappa}{Ca^*}\left(\frac{\partial \Phi_0}{\partial z}\right)^2.
\end{align*}
Integrating in $z$ from $-\infty$ to $+\infty$ gives
\begin{equation}
  [P_0]_{-\infty}^{+\infty}
  = \frac{\kappa}{Ca^{\ast}} \int_{-\infty}^{\infty}
    \left(\frac{d\Phi_0}{dz}\right)^2 dz.
\end{equation}
We introduce
\[
  \lambda =
  \int_{-\infty}^{\infty}
  \left(\frac{d\Phi_0}{dz}\right)^2 dz,
\]
so that
\[
  [P_0]_{-\infty}^{+\infty}
  = \frac{\lambda}{Ca^{\ast}}\,\kappa.
\]
Using the matching condition \eqref{matching5}, we finally obtain the
Young--Laplace law
\begin{equation}  
  \left[\tilde{p}_0\right]_{1}^{2}
  = \sigma \kappa
  \quad \text{on } \Sigma,
  \qquad
  \sigma = \frac{\lambda}{Ca^{\ast}}.
\end{equation}

\subsection{Sharp interface limit}
In summary of the outer expansion and the boundary layer expansion,  we obtain the sharp interface model as follows
\begin{subequations}
    \begin{align}
        \mathbb{K}^{-1}\boldsymbol{u_0} &= -\nabla p_0  -\frac{Bo}{Ca^\ast}\rho_0 \boldsymbol{k} \quad &\text{in } \, \Omega^+\cup \Omega^-,\\
        \nabla \cdot \bld{u_0}&=0 &\text{ in } \Omega^+\cup \Omega^-,\\
        \left[\boldsymbol{u}_0\right]_1^2\cdot\boldsymbol{\nu}&=0 & \text{ on } \Sigma,\\
        \boldsymbol{V^\nu}_0 &= \frac{1}{\chi} \bld{u_0}\cdot \boldsymbol{\nu}. &\text{ on } \Sigma,\\
        \left[p_0\right]_{1}^{2} &= \sigma \kappa  &\text{ on } \Sigma,
    \end{align}
\end{subequations}
    Where $\tilde{p}_0$ is written as $p_0$ for simplicity.

\section{Conclusion}	

The qCHD model developed in this paper describes two-phase flow in porous media for unmatched densities. This model maintains the thermodynamic structure of the diffuse-interface formulations. Also, the global existence of weak solution is established too. The proof is based on a pressure estimate from Darcy's equation, which allowed the control over the coupling terms arising from quasi-incompressibility and the singular logarithmic potential. Additionally, the sharp-interface asymptotics further demonstrate that this model is consistent with the Muskat-type free-boundary problem. All these findings provide a solid mathematical foundation for the qCHD model and encourage further study into strong existence, weak-strong uniqueness, long-time behavior, and stable, higher-order numerical approximations.

\section*{Acknowledgement}
D. Han and S. Sarkar acknowledge support from the National Science Foundation under grant DMS-2310340.

\section*{Financial disclosure}
None reported.

\section*{Conflict of interest}
The authors declare no potential conflict of interests.

	\bibliographystyle{unsrturl}
\bibliography{sayantan-ref}
	
\end{document}